\documentclass{amsart}
\usepackage{graphicx}
\newtheorem{theorem}{Theorem}[section]
\newtheorem{lemma}[theorem]{Lemma}
\newtheorem{corollary}[theorem]{Corollary}
\newtheorem{proposition}[theorem]{Proposition}
\newtheorem{problem}[theorem]{Problem}

\theoremstyle{definition}

\theoremstyle{remark}

\numberwithin{equation}{section}

\begin{document}

\title{Non-minimal bridge positions of torus knots are stabilized}

\author{Makoto Ozawa}
\address{Department of Natural Sciences, Faculty of Arts and Sciences, Komazawa University, 1-23-1 Komazawa, Setagaya-ku, Tokyo, 154-8525, Japan}
\curraddr{Department of Mathematics and Statistics, The University of Melbourne
Parkville, Victoria 3010, Australia (temporary)}
\email{w3c@komazawa-u.ac.jp, ozawam@unimelb.edu.au (temporary)}

\subjclass[2000]{Primary 57M25; Secondary 57Q35}



\keywords{torus knot, bridge decomposition, bridge position, bridge presentation, stabilization, stabilized, unstabilized, stably equivalent}

\begin{abstract}
We show that any non-minimal bridge decomposition of a torus knot is stabilized and that $n$-bridge decompositions of a torus knot are unique for any integer $n$.
This implies that a knot in a bridge position is a torus knot if and only if there exists a torus containing the knot such that it intersects the bridge sphere in two essential loops.
\end{abstract}

\maketitle

\section{Introduction}

Throughout this paper we work in the piecewise linear category.

Knot theory treats all embeddings of the circle $S^1$ into the 3-sphere $S^3$ up to equivalence relation by orientation preserving homeomorphisms of $S^3$, and it is most fundamental and important problem to determine whether given two knots are equivalent and furthermore to describe how one can be deformed to another.

In \cite{Sc}, Schubert introduced the bridge number of knots, that is, the half of minimal number of intersection of a 2-sphere $S$ with a knot $K$ which cuts the pair $(S^3,K)$ into two trivial tangles $(B_1,T_1)$ and $(B_2,T_2)$.
Since a knot is trivial if and only if it has the bridge number $1$, the bridge number measures some complexity of knots.

Two bridge decompositions $(B_1,T_1)\cup_{S}(B_2,T_2)$ and $(B_1',T_1')\cup_{{S}'}(B_2',T_2')$ of a pair $(S^3,K)$ are said to be {\em equivalent} if there exists an orientation preserving homeomorphism of $S^3$ which sends one to another.
Since we may assume that $S={S}'$, $B_1=B_1'$ and $B_2=B_2'$, this condition can be replaced by the following condition: there exist homeomorphisms $f_1$ of $B_1$ and $f_2$ of $B_2$ such that $f_1(T_1)=T_1'$, $f_2(T_2)=T_2'$ and $f_1|_{S}=f_2|_{S}$.
Therefore, the bridge decomposition has a merit in a sense that it can decompose an orientation preserving homeomorphism of $S^3$ into two homeomorphisms of $B_1$ and $B_2$ if the knot admits a unique $n$-bridge decomposition.

In exchange for this, bridge decompositions of a knot $K$ are not unique generally even if it has the minimal bridge number.
For example, Birman (\cite{B}) and Montesinos (\cite{Mo}) gave an example of composite knots and prime knots respectively which admit at least two equivalence classes of minimal bridge decompositions.
The following theorem seems to be well-known since it can be proved by showing that Reidemeister moves (\cite{AB}, \cite{R}) can be achieved by bridge isotopies and stabilizations (this has been also pointed out by Joel Hass).

\begin{theorem}[well-known]
All bridge decompositions of a knot are stably equivalent.
\end{theorem}

This theorem says that any two bridge decompositions are equivalent after some stabilizations.
(The term ``stabilization'' in \cite{B}, \cite{O1}, etc is exchange for ``perturbation'' in \cite{ST}, \cite{T}, etc. since they studied a generalized bridge decomposition of knots in a 3-manifold and the term``stabilization'' has been already used for Heegaard splittings.)
This result corresponds to Reidemeister-Singer's theorem \cite{R}, \cite{S}: all Heegaard splittings of a closed orientable 3-manifold are stably equivalent.
And also, this is compared to Markov theorem \cite{M}: all closed braid representatives of a knot are stably equivalent.

\begin{figure}[htbp]
	\begin{center}
	\includegraphics[trim=0mm 0mm 0mm 0mm, width=.7\linewidth]{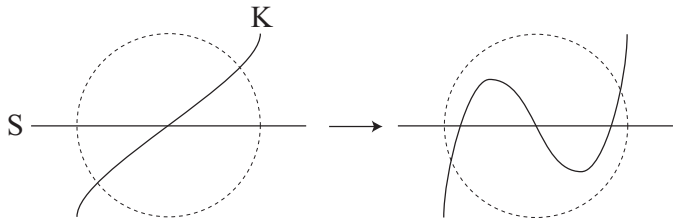}
	\end{center}
	\caption{Stabilization (Perturbation)}
	\label{stabilization}
\end{figure}

Toward the unknotting problem, Otal (later Hayashi and Shimokawa, the author) showed the following theorem.

\begin{theorem}[\cite{O1}, \cite{HS}, \cite{OZ}]\label{Otal}
Any non-minimal bridge decomposition of the trivial knot is stabilized.
\end{theorem}

This theorem says that any two $n$-bridge decompositions of the trivial knot are equivalent since the trivial knot admits a unique 1-bridge decomposition.
Thus, a knot $K$ with an $n$-bridge decomposition $(B_1,T_1)\cup_{S}(B_2,T_2)$ is trivial if and only if there exists a 2-sphere $F$ containing $K$ such that $F\cap S$ consists of a single loop (\cite{OZ}).
This result corresponds to Waldhausen's theorem \cite{W}: any non-minimal Heegaard splitting of the 3-sphere is stabilized, and by Alexander's theorem \cite{A}, any two genus $n$ Heegaard splittings are equivalent.

Furthermore, Otal (later Scharlemann and Tomova) showed the following theorem.

\begin{theorem}[\cite{O2}, \cite{ST}]\label{2-bridge}
Any non-minimal bridge decomposition of a 2-bridge knot is stabilized.
\end{theorem}

This theorem says that any two $n$-bridge decompositions of a 2-bridge knot are equivalent up to reflection since a 2-bridge knot admits at most two 2-bridge decompositions and they are related by the reflection with respect to the bridge decomposing 2-sphere.
This result corresponds to Bonahon-Otal's theorem \cite{BO}: any non-minimal Heegaard splitting of a lens space is stabilized.

In this paper, we show the following theorem.

\begin{theorem}\label{torus}
Any non-minimal bridge decomposition of a torus knot is stabilized.
\end{theorem}

We remark that $b(K)=\min \{p,q\}$ for a $(p,q)$-torus knot $K$ $(p,q>0)$, where $b(K)$ denotes the bridge number of $K$.
This was proved by Schubert \cite{Sc} and later Schultens \cite{Sch}, the author \cite{OZ}.
Furthermore, $n$-bridge decompositions of a torus knot are unique for any integer $n$ since a torus knot admits a unique minimal bridge decomposition.

\begin{corollary}
A knot $K$ with an $n$-bridge decomposition $(B_1,T_1)\cup_{S} (B_2,T_2)$ is a torus knot if and only if there exists a torus $F$ containing $K$ such that $F\cap S$ consists of two essential loops.
\end{corollary}

It is clear that if a bridge decomposition of a knot is minimal, then it is unstabilized.
Theorem \ref{Otal}, \ref{2-bridge} and \ref{torus} show that the converse holds for the trivial knot, 2-bridge knots and torus knots respectively.

\begin{problem}
Is any non-minimal bridge decomposition of a knot stabilized?
\end{problem}

In this respect, Casson and Gordon \cite{CG} (later Kobayashi \cite{K}, Lustig and Moriah \cite{LM}, Moriah, Schleimer and Sedgwick \cite{MSS}) showed that there exists a 3-manifold which admits unstabilized Heegaard splittings of arbitrarily large genus.
Therefore it is expected that there exists a knot which admits non-minimal unstabilized bridge decompositions.

Theorem \ref{Otal} and \ref{torus} show that for any integer $n$, there exists a unique $n$-bridge decomposition of the trivial knot and a torus knot respectively, and Theorem \ref{2-bridge} shows that for any integer $n$, there exist at most two $n$-bridge decompositions of a 2-bridge knot.
Generally, Coward showed the following theorem.

\begin{theorem}[\cite{C}]
There exist only finitely many bridge decompositions of given bridge numbers for a hyperbolic knot.
\end{theorem}

This result corresponds to Li's settlement \cite{L} of the Waldhausen conjecture: there exist only finitely many Heegaard splittings of given Heegaard genus for a closed, orientable, irreducible and atoroidal 3-manifold.

In despite of this, Jang \cite{J} showed that there exists a 3-bridge link which admits infinitely many 3-bridge decompositions.
And also Sakuma \cite{Sa} (later Morimoto and Sakuma \cite{MS}, Bachman and Derby-Talbot \cite{BD}) showed that there exists a 3-manifold which admits infinitely many minimal genus Heegaard splittings.
We also remark that Jang \cite{J2} classified 3-bridge algebraic links and moreover, for non-Montesinos ones, she also given a classification of 3-bridge decompositions of each link.

\section{Definitions and related results}

\subsection{Three definitions of the bridge number}



Let $K$ be a knot in the 3-sphere $S^3$, and $F$ be a 2-sphere embedded in $S^3$ which intersects $K$ in $2n$ points $(n\ge1)$.
Then by the Alexander's theorem \cite{A}, $F$ decomposes $S^3$ into two 3-balls $B_1$ and $B_2$, and $K$ into two collections $T_1$ and $T_1$ of $n$ arcs.
Then the pair $(B_i,T_i)$ is called an {\em $n$-string tangle}, and we say that the pair $(S^3,K)$, or simply $K$, admits an {\em $n$-string tangle decomposition} $(B_1,T_1)\cup_F (B_2,T_2)$.
We say that a tangle $(B_i,T_i)$ is {\em trivial} if there exists a disk $D$ properly embedded in $B_i$ such that $D\supset T_i$, and that $K$ admits a {\em trivial $n$-string tangle decomposition}, or simply an {\em $n$-bridge decomposition}, $(B_1,T_1)\cup_F (B_2,T_2)$ if both of $(B_1,T_1)$ and $(B_2,T_2)$ are trivial.
The {\em bridge number} $b(K)$ is defined as the minimal number of $n$ such that $K$ admits an $n$-bridge decomposition.

There are at least two other definitions of the bridge number as follows.

First suppose that a knot $K$ is disjoint from $\pm\infty=(0,0,0,\pm 1)\in \Bbb{R}^4$ and let $p:S^3-\{\pm\infty\} \to S$ be a projection by regarding $S^3-\{\pm\infty\}$ as $S\times \Bbb{R}$.
Then we say that $p(K)$ is a {\rm regular projection} if it has only transverse double points, and a {\em regular diagram} $\tilde{K}$ of $K$ is a regular projection with an over/under information on its double points, namely {\em crossings}.
We say that a subarc of $\tilde{K}$ is an {\em over bridge} (resp. {\em under bridge}) if it contains only over crossings (resp. under crossings).
A partition $\tilde{K}=K_1^+\cup K_1^-\cup \cdots \cup K_n^+\cup K_n^-$ is called an {\em $n$-bridge presentation} of $K$ if $K_i^+$ is an over bridge and $K_i^-$ is an under bridge for $i=1,\ldots,n$.
The {\em bridge number} $b(K)$ is defined as the minimal number of $n$ such that $K$ admits an $n$-bridge presentation.

Next let $h:S^3\to \Bbb{R}$ be the standard Morse function, that is, the restriction of $\Bbb{R}^4=\Bbb{R}^3\times \Bbb{R}\to \Bbb{R}$ on $S^3$.
We may assume that $K$ is disjoint from $\pm\infty$ and $K$ has only finitely many critical points with respect to $h$.
We say that $K$ is in an {\em $n$-bridge position} if $K$ has just $n$ maximal points and $n$ minimal points, and all maximal (resp. minimal) values are greater (resp. less) than $0$.
Then the 2-sphere $S=h^{-1}(0)$ is called an {\em $n$-bridge sphere} for $K$.

It is well known that the above three definitions are equivalent.
We do not prove this here, but Scharlemann's survey \cite{S1} is a good reference.

\begin{proposition}\label{fact}
The following are equivalent.
\begin{enumerate}
\item A knot $K$ admits an $n$-bridge decomposition with a bridge sphere $S$.
\item A knot $K$ admits an $n$-bridge presentation with a bridge sphere $S$.
\item A knot $K$ admits an $n$-bridge position with a bridge sphere $S$.
\end{enumerate}
\end{proposition}

\subsection{On the connected sum of knots}

Let $K_i$ $(i=1,2)$ be a knot in $S^3$ which admits an $n_i$-bridge decomposition $(B_i^+,T_i^+)\cup_{F_i}(B_i^-,T_i^-)$.
Pick one point $p_i$ of $K_i\cap F_i$ and take a small regular neighborhood $C_i$ of $p_i$ so that the pair $(C_i,K_i\cap C_i)$ is a trivial 1-string tangle.
Then we obtain an $(n_1+n_2-1)$-bridge decomposition $((B_1^+-\text{int}C_1)\cup(B_2^+-\text{int}C_2), (T_1^+-\text{int}(K_1\cap C_1))\cup (T_2^+-\text{int}(K_2\cap C_2))) \cup_{(F_1-\text{int}(F_1\cap C_1)\cup(F_2-\text{int}(F_2\cap C_2))} ( (B_1^--\text{int}C_1)\cup(B_2^--\text{int}C_2), (T_1^--\text{int}(K_1\cap C_1))\cup (T_2^--\text{int}(K_2\cap C_2)))$ for a connected sum $K_1\# K_2$ of $K_1$ and $K_2$, and we denote it by $(B_1^+,T_1^+)\cup_{F_1}(B_1^-,T_1^-)\# (B_2^+,T_2^+)\cup_{F_2}(B_2^-,T_2^-)$.
Thus, we have $b(K_1\# K_2)\le b(K_1)+b(K_2)-1$.
Schubert (later Schultens) showed the converse inequality: $b(K_1\# K_2)\ge b(K_1)+b(K_2)-1$.
The proof involves the following theorem.

\begin{theorem}[\cite{Sc}, \cite{Sch1}, \cite{D}, \cite{HS2}]
For any $n$-bridge decomposition $(B_1,T_1)\cup_{S} (B_2,T_2)$ of a composite knot $K$, there exists a decomposing sphere $F$ for $K$ such that $F\cap S$ consists of a single loop.
\end{theorem}

This theorem says that any bridge decomposition of a composite knot $K_1\# K_2$ can be decomposed into two bridge decompositions of $K_1$ and $K_2$.
This result corresponds to Haken's theorem \cite{H}: for any Heegaard splitting of a reducible closed orientable 3-manifold, there exists an essential 2-sphere which intersects the Heegaard surface in a sigle loop.
This shows that the Heegaard genus is additive under connected sum.
And also, this is compared to Birman-Menasco's theorem \cite{BM}: for any closed braid representatives of a composite knot $K$, there exists a decomposing sphere $F$ for $K$ such that $F$ intersects the braid axis in two points after some exchange moves.
Since an exchange move does not change the braid index, this shows that the braid index $(-1)$ is additive under connected sum.

\subsection{Stabilization and stably equivalent}
Now we define the {\em stabilization} and {\em stably equivalent} of bridge decompositions as follows.
Let $K_0$ be the trivial knot which admits a 2-bridge decomposition $(B_0^+,T_0^+)\cup_{F_0}(B_0^-,T_0^-)$.
It follows by Theorem \ref{Otal} that there exists a 2-sphere which contains the trivial knot and intersects $F_0$ in a single loop.
The {\em stabilization} of an $n$-bridge decomposition $(B^+,T^+)\cup_{F}(B^-,T^-)$ of a knot $K$ is an operation to obtain a new $(n+1)$-bridge decomposition $(B^+,T^+)\cup_{F}(B^-,T^-)\# (B_0^+,T_0^+)\cup_{F_0}(B_0^-,T_0^-)$ of $K\# K_0$ from $(B^+,T^+)\cup_{F}(B^-,T^-)$, and the resultant bridge decomposition is said to be {\em stabilized}.
See Figure \ref{stabilization}.
We say that a bridge decomposition is {\em unstabilized} if it is not stabilized.
Two bridge decompositions $(B_1^+,T_1^+)\cup_{F_1}(B_1^-,T_1^-)$ and $(B_2^+,T_2^+)\cup_{F_2}(B_2^-,T_2^-)$ are {\em stably equivalent} if they are equivalent after some stabilizations, namely, there exist non-negative integers $m,n$ such that $\displaystyle (B_1^+,T_1^+)\cup_{F_1}(B_1^-,T_1^-)\#^m (B_0^+,T_0^+)\cup_{F_0}(B_0^-,T_0^-)$ and $(B_2^+,T_2^+)\cup_{F_2}(B_2^-,T_2^-)\#^n (B_0^+,T_0^+)\cup_{F_0}(B_0^-,T_0^-)$ are equivalent.

\subsection{Weakly reducible bridge decompositions}

Casson and Gordon \cite{CG2} showed that if a Heegaard splitting of a closed orientable 3-manifold $M$ is weakly reducible, then either it is reducible or $M$ contains an essential surface.
This result was generalized by Hayashi and Shimokawa (later Tomova) to a 3-manifold containing a 1-manifold as follows.
A bridge decomposition $(B^+,T^+)\cup_{F}(B^-,T^-)$ of a knot $K$ is {\em weakly reducible} if there exists a pair of compressing disks $D^{\pm}$ for $F-K$ in $B^{\pm}-T^{\pm}$ such that $D^+\cap D^-=\emptyset$, and otherwise it is {\em strongly irreducible}.
A properly embedded surface $S$ in the exterior $E(K)$ of $K$ is an {\em essential meridional surface} if $\partial S$ consists of meridians of $K$ and $S$ is incompressible and $\partial$-incompressible in $E(K)$.
The following theorem has been reduced to bridge positions of a knot in $S^3$.

\begin{theorem}[\cite{HS2}, \cite{T}]\label{WR}
If a bridge position of a knot is weakly reducible, then either it is stabilized or the knot exterior contains an essential meridional surface.
\end{theorem}

\subsection{Essential Morse bridge position}

Following \cite{OZ}, we define a {\em Morse bridge position} of a pair $(F,K)$ of a closed surface $F$ containing a knot $K$ with respect to $h:S^3\to \Bbb{R}$ as follows.
We divide $S^3$ into three parts $M_+$, $M_0$ and $M_-$ so that $M_{\pm}\cap M_0=h^{-1}(c_{\pm})$ is a level 2-sphere for $c_+>0>c_-$.
We say that a pair $(F,K)$ with a bridge decomposition $(B^+,T^+)\cup_{S}(B^-,T^-)$ is in a {\em Morse bridge position} if

\begin{enumerate}
\item $h|_F$ is a Morse function,
\item $K$ is disjoint from critical points of $F$,
\item all maximal (resp. minimal) points of $F$ are contained in $M_+$ (resp. $M_-$),
\item all saddle points of $F$ are contained in $M_0$,
\item $(M_{\pm},K\cap M_{\pm})$ is homeomorphic to $(B^{\pm},T^{\pm})$ as a pair,
\item $(M_0,K\cap M_0)$ is homeomorphic to $(S,K\cap S)\times I$ as a pair,
\item $S=h^{-1}(0)$.
\end{enumerate}


Let $x$ be a saddle point of $F$ which corresponds to the critical value $t_x\in \Bbb{R}$.
Let $P_x$ be a pair of pants component of $F\cap h^{-1}([t_x-\epsilon, t_x+\epsilon])$ containing $x$ for a sufficiently small positive real number $\epsilon$.
Let $C_x^1$, $C_x^2$ and $C_x^3$ be the boundary components of $P_x$, where we assume that $C_x^1$ and $C_x^2$ (``thigh'' loops) are contained in the same level $h^{-1}(t_x\pm\epsilon)$, and $C_x^3$ (a ``waist'' loop) is contained in the another level $h^{-1}(t_x\mp\epsilon)$.
A saddle point $x$ of $F$ is {\em upper} (resp. {\em lower}) if $C_x^1$ and $C_x^2$ are contained in $h^{-1}(t_x-\epsilon)$ (resp. $h^{-1}(t_x+\epsilon)$).
See Figure \ref{upper/lower}.
A saddle point $x$ of $F$ is {\em essential} if both of the two loops $C_x^1$ and $C_x^2$ are essential in $F$, and it is {\em inessential} if it is not essential.
A Morse bridge position of a pair $(F,K)$ is said to be {\em essential} if $F$ has no inessential saddle point.

\begin{figure}[htbp]
	\begin{center}
	\begin{tabular}{ccc}
	\includegraphics[trim=0mm 0mm 0mm 0mm, width=.35\linewidth]{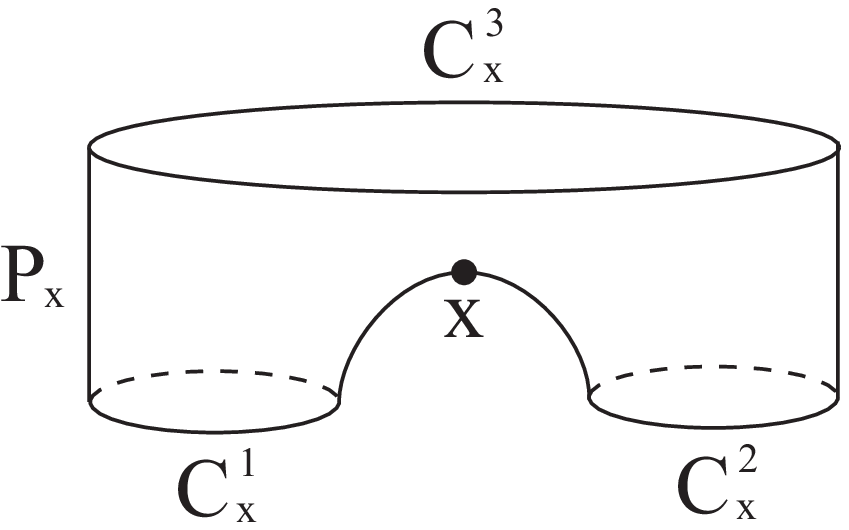}&
	\includegraphics[trim=0mm 0mm 0mm 0mm, width=.35\linewidth]{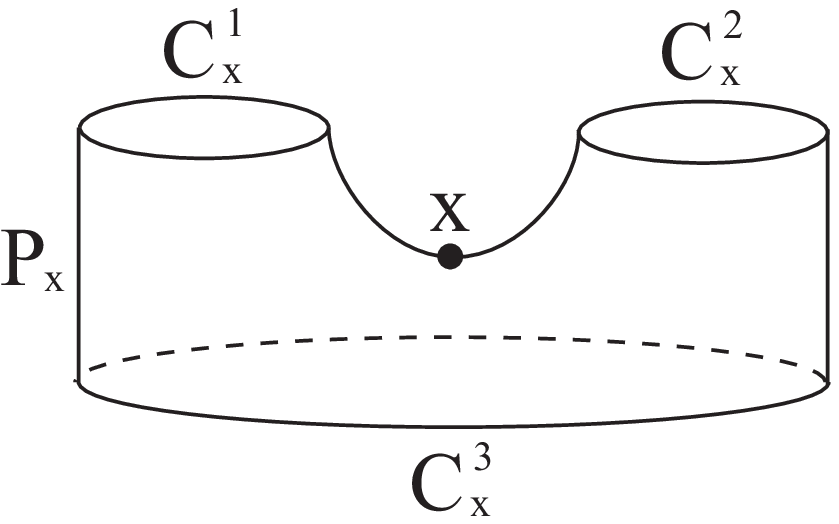}\\
	Upper saddle & Lower saddle
	\end{tabular}
	\end{center}
	\caption{A pair of pants $P_x$ containing a saddle point $x$}
	\label{upper/lower}
\end{figure}

\begin{lemma}[\cite{OZ},\cite{Sch}]\label{essential saddle}
A pair $(F,K)$ of a closed surface $F$ containing a knot $K$ with a bridge decomposition $(B^+,T^+)\cup_{S}(B^-,T^-)$ can be isotoped in its equivalence class so that it is in an essential Morse bridge position.
\end{lemma}

\section{Proof of Theorem \ref{torus}}
\begin{proof}
Let $K$ be a $(p,q)$-torus knot $(2\le p<q)$ in $S^3$ which admits a non-minimal $n$-bridge decomposition $(B^+,T^+)\cup_{S}(B^-,T^-)$ $(p<n)$.
Since $K$ is a torus knot, there exists a torus $F$ containing $K$ such that $F$ separates $S^3$ into two solid tori $W_1$ and $W_2$, where we assume that $W_1$ contains a meridian disk $w_1$ with $|\partial w_1\cap K|=p$ and $W_2$ contains a meridian disk $w_2$ with $|\partial w_2\cap K|=q$.

\begin{lemma}
If a bridge decomposition of a torus knot is weakly reducible, then it is stabilized.
\end{lemma}

\begin{proof}
Since the torus knot exterior does not contain an essential meridional surface \cite{Ts}, this lemma follows Theorem \ref{WR}.
\end{proof}

Hereafter, we assume that $(B^+,T^+)\cup_{S}(B^-,T^-)$ is strongly irreducible.

Proposition \ref{fact} shows that $K$ admits an $n$-bridge position with respect to the hight function $h:S^3\to \Bbb{R}$, where $S=h^{-1}(0)$, $B^{\pm}=h^{-1}([0,\pm\infty))$.
Then we apply Lemma \ref{essential saddle} to the pair $(F,K)$.
Since $F$ is a torus, for each essential saddle point $x$ of $F$, the waist loop $C_x^3$ bounds a disk $D_x$ in $F$ which contains only one maximal or minimal point, and the thigh loops $C_x^1$ and $C_x^2$ are mutually parallel essential loops in $F$.
Since all thigh loops are parallel to either $\partial w_1$ or $\partial w_2$ in $F$, each disk $D_x$ contains at least $p$ maximal (resp. minimal) points of $K$, where $x$ is an upper (resp. lower) saddle point.
Two saddle points $x$ and $y$ of $F$ are said to be {\em adjacent} if one of $C_x^1, C_x^2$ and one of $C_y^1, C_y^2$ cobound an annulus in $F$, say $A_{xy}$, which does not contain any critical point.
We call such an annulus $A_{xy}$ an {\em adjoining annulus} between $x$ and $y$.
With this notation, we have a partition $F=(\bigcup_{x\in X} P_x) \cup (\bigcup_{x\in X} D_x) \cup (\bigcup_{x,y\in X} A_{xy})$ of pairs of pants, disks and adjoining annuli, where $X$ is the set of all saddle points of $F$.

The following lemma says that the torus $F$ is also in a ``bridge position'', that is, all upper saddle points are above all lower saddle points.

\begin{lemma}\label{adjoining}
There exists a level 2-sphere $h^{-1}(t)$ $(c_+>t>c_-)$ in $M_0$ which intersects all adjoining annuli.
\end{lemma}

\begin{proof}
This lemma follows the argument of \cite[Lemma 6]{Sch2} and the assumpution that $(B^+,T^+)\cup_{S}(B^-,T^-)$ is strongly irreducible.
If all upper saddle points are above all lower saddle points, then this lemma holds.
Otherwise, there exists a regular value $t_0\in \Bbb{R}$ between adjacent critical values $t_x>t_y$ which correspond to a lower saddle point $x$ and an upper saddle point $y$ respectively, in other words, $h^{-1}(t_0)$ is a ``thin sphere'' for $F$.
Then $D_x'=D_x\cap h^{-1}([t_0,-\infty))$ and $D_y'=D_y\cap h^{-1}([t_0,+\infty))$ are disjoint disks properly embedded in $h^{-1}([t_0,-\infty))$ and $h^{-1}([t_0,+\infty))$ respectively.
Since both of $D_x'$ and $D_y'$ contain at least $p$ minimal or maximal points respectively, if we isotope $D_x'$ and $D_y'$ slightly, then we have two disks which shows that a bridge decomposition given by $h^{-1}(t_0)$ is weakly reducible.
By the condition (6) of Morse bridge position, it follows that $(B^+,T^+)\cup_{S}(B^-,T^-)$ is also weakly reducible.
This contradicts the assumpution.
\end{proof}

By the condition (6) of Morse bridge position, we may assume that the level 2-sphere $h^{-1}(t)$ given in Lemma \ref{adjoining} is $S=h^{-1}(0)$.
Thus $F\cap B^{\pm}$ consists of annuli each of which contains only one essential saddle point and only one maximal or minimal point.
Hereafter we denote this annulus decomposition of $F$ by $F=A_1\cup A_2\cup \cdots \cup A_n$ in this order on the torus $F$, where $n=|F\cap S|$ and $A_i\subset B^+$ (resp. $A_i\subset B^-$) if $i$ is odd (resp. even).

\subsection{The case that $|F\cap S|\ge 4$}

First we consider the case that $|F\cap S|\ge 4$.
Put $C_i=A_i\cap A_{i+1}$ for $i=1,\ldots,n-1$ and $C_n=A_n\cap A_1$.
Let $A_i'$ be an annulus in $S$ which is cobounded by $\partial A_i$ for $i=1,\ldots,n$.

There are three possibilities for a pair of adjacent annuli $(A_i,A_{i+1})$ $(i=1,\ldots,n-1)$ or $(A_n,A_1)$.

\begin{enumerate}
\item ``{\em Normal}'': Neither turning nor Reeb.
\item ``{\em Turning}'': $\text{int}A_i'\cap \text{int}A_{i+1}'=\emptyset$ $(i=1,\ldots,n-1)$ or $\text{int}A_n'\cap \text{int}A_1'=\emptyset$.
\item ``{\em Reeb}'': $A_i'\subset A_{i+1}'$ or $A_i'\supset A_{i+1}'$ $(i=1,\ldots,n-1)$, or $A_n'\subset A_1'$ or $A_n'\supset A_1'$.
\end{enumerate}

\begin{figure}[htbp]
	\begin{center}
	\begin{tabular}{ccc}
	\includegraphics[trim=0mm 0mm 0mm 0mm, width=.35\linewidth]{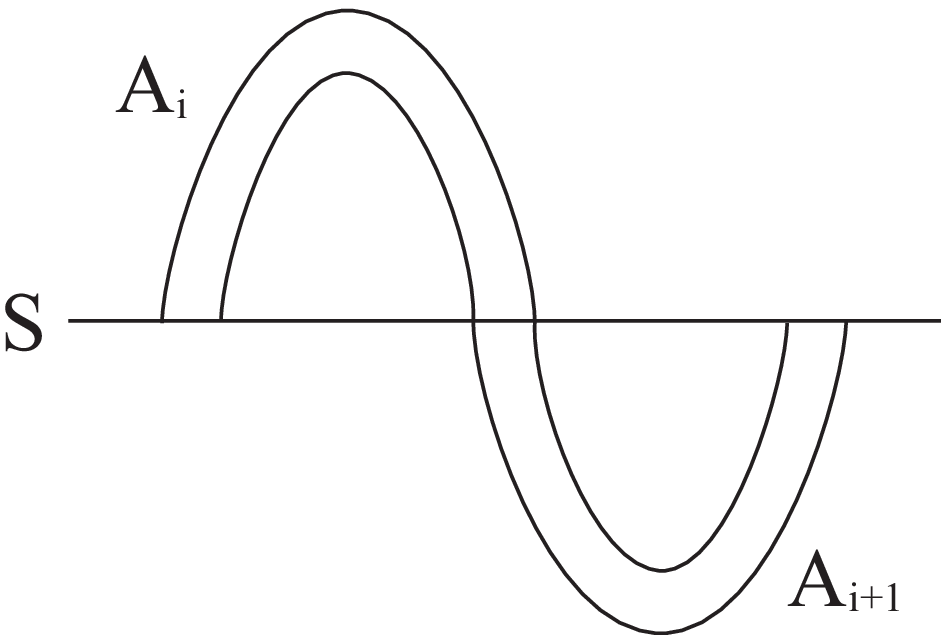}&
	\includegraphics[trim=0mm 0mm 0mm 0mm, width=.32\linewidth]{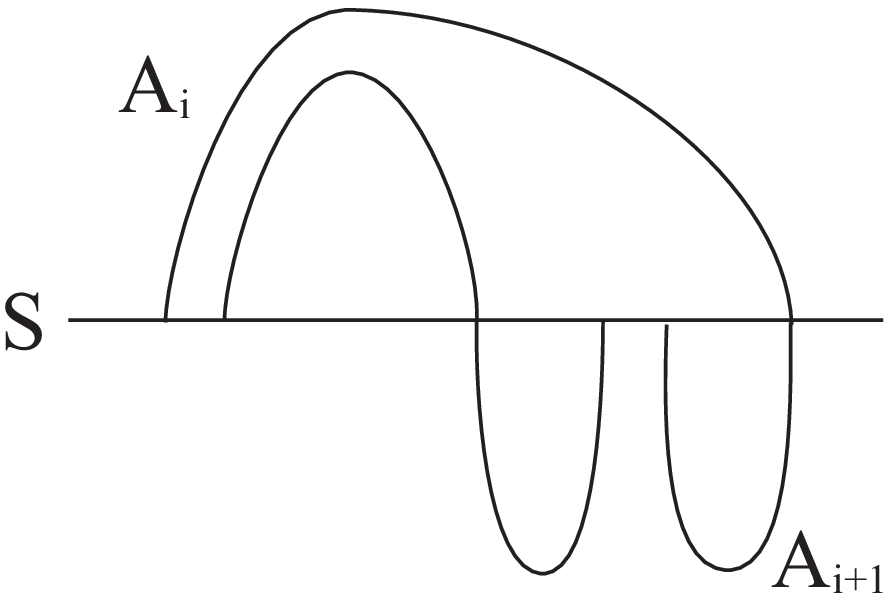}&
	\includegraphics[trim=0mm 0mm 0mm 0mm, width=.25\linewidth]{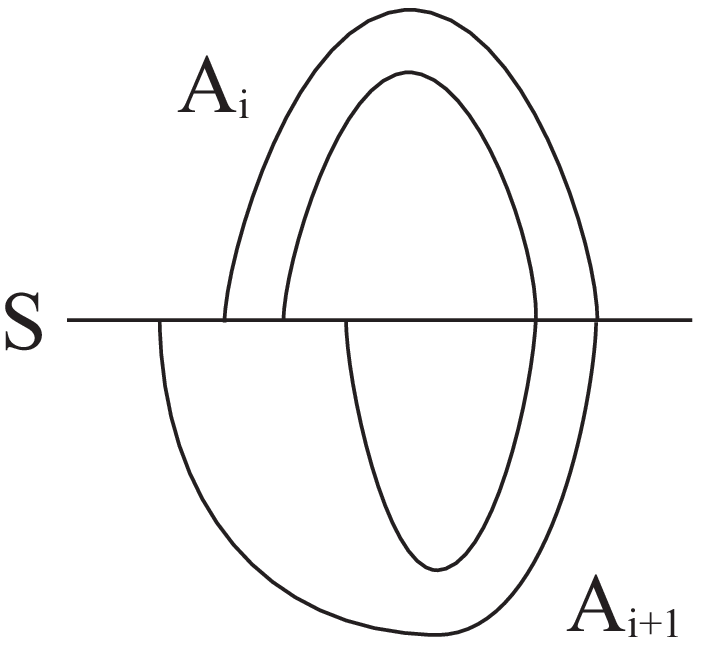}\\
	Normal & Turning & Reeb
	\end{tabular}
	\end{center}
	\caption{Adjacent annuli}
	\label{adjacent}
\end{figure}

\begin{lemma}\label{two cases}
Then one of the following holds.
\begin{enumerate}
\item $C_1,\ldots,C_n$ bound mutually disjoint disks $D_1,\ldots,D_n$ in $S$.
\item There exists a turning pair of adjacent annuli.
\end{enumerate}
\end{lemma}

\begin{proof}
We suppose that there exists no turning pair of adjacent annuli, and show that $C_1,\ldots,C_n$ bound mutually disjoint disks $D_1,\ldots,D_n$ in $S$.

If there exists a Reeb pair of adjacent annuli $(A_i,A_{i+1})$ such that $A_i'\supset A_{i+1}'$ without loss of generality, then $(A_{i+1},A_{i+2})$ is also a Reeb pair of adjacent annuli such that $A_{i+1}'\supset A_{i+2}'$ since $C_{i+1}$ is essential in $A_i'$ and there exists no turning pair of adjacent annuli.
As a consequence of this observation, we obtain an infinite sequence $A_i'\supset A_{i+1}'\supset A_{i+2}'\supset \cdots$, a contradiction.
Hence, all pair of adjacent annuli are normal.

We show by an induction on $k$ $(k=1,\ldots,n-1)$ that $C_n,C_1,\ldots,C_k$ bound mutually disjoint disks $D_n,D_1,\ldots,D_k$ in $S$ whose interiors are disjoint from $A_1\cup \cdots \cup A_k$.
We may assume without loss of generality that $C_n$ is innermost in $S$ among all $C_1,\ldots,C_n$.
First, it is clear that $C_n,C_1$ bound mutually disjoint disks $D_n,D_1$ in $S$ whose interiors are disjoint from $A_1$.
Thus the induction holds on $k=1$.
Next suppose that the induction holds for $k<i$ and does not hold on $k=i$.
Then $C_i$ bounds a disk $D_i$ in $S$ with $D_i\cap D_{i-1}=\emptyset$ such that $D_i$ contains some disks of $D_n,D_1,\ldots,D_{i-3}$.
This implies that there exists a ``long Reeb''$A_j\cup A_{j+1}\cup \cdots \cup A_{j+s-1}$ of length $s$ $(1\le j\le i-3,\ j+s-1=i)$, that is, a subsequence of annuli between $A_j$ and $A_i$ such that $D_i\supset D_{j-1}$ if $2\le j\le i-3$ or $D_i\supset D_n$ if $j=1$.
Since there exists no turning pair of adjacent annuli, we have $A_j'\supset A_{i+1}',\ A_{j+1}'\supset A_{i+2}',\ \ldots,\ A_{j+s-1}'\supset A_{i+s}'$.
Thus $A_i'\supset A_{i+s}'$ and hence $A_{i+1}\cup\cdots\cup A_{i+s}$ is also a long Reeb of length $s$.
Consequently, we obtain an infinite sequence $A_i'\supset\cdots\supset A_{i+s}'\supset\cdots\supset A_{i+2s}'\supset \cdots $, a contradiction.
\end{proof}

In the conclusion (1) of Lemma \ref{two cases}, there exists a core loop $L$ of the solid torus $W_i$ $(i=1$ or $2)$ such that $L$ intersects $D_i$ in one point for $i=1,\ldots,n$.
Then we have an $n/2$-bridge decomposition of a trivial knot $L$ by $S$ $(n\ge4)$ and it is stabilized by Theorem \ref{Otal}.
This implies that an $n$-bridge decomposition $(B^+,T^+)\cup_{S}(B^-,T^-)$ of $K$ is also stabilized.

In the conclusion (2) of Lemma \ref{two cases}, without loss of generality, we may assume that $(A_1,A_2)$ is a turning pair of adjacent annuli.
Let $E_1\subset B^+$ and $E_2\subset B^-$ be $\partial$-compressing disks for $A_1$ and $A_2$ respectively.

\begin{lemma}\label{essential arcs}
The $\partial$-compressing disk $E_j$ $(j=1,2)$ can be isotoped so that it intersects $A_i$ in essential arcs $(i=3,\ldots,n)$.
\end{lemma}

\begin{proof}
Without loss of generality, we show this lemma only for $E_1$.

We remark first that each annulus $A_i$ $(i=3,5,\ldots,n-1)$ separates the 3-ball $B^+$ into a 3-ball, say $B_i$, and a solid torus, say $V_i$, and $A_i$ is $\partial$-parallel to $A_i'$ in $V_i$ since $A_i$ is a ``1-bridge annulus'', that is, it contains only one essential saddle point and only one maximal point.
We remark next that there are the following three types for an annulus $A_i$ in $V_1$.
\begin{enumerate}
\item $\partial A_i$ consists of mutually non-parallel inessential loops in $A_1'$.
\item $\partial A_i$ consists of mutually parallel inessential loops in $A_1'$.
\item $\partial A_i$ consists of mutually parallel essential loops in $A_1'$.
\end{enumerate}
In this three types, an annulus of type (1) is not $\partial$-parallel to $A_i'$ in $V_1$, and an annulus of type (2) or (3) is $\partial$-parallel to $A_i'$ in $V_1$.
Let $\mathcal{A}$ be the set of annuli contained in $V_1$ from $A_i$ $(i=3,5,\ldots,n-1)$, and $\partial \mathcal{A}\subset \mathcal{A}$ be the set of annuli which are $\partial$-parallel to $A_1'$, in other words, an annulus $A_i$ which satisfies $A_i'\subset A_1'$.

We show this lemma by an induction on $l=|\partial \mathcal{A}|$.

First if $l=0$, then all annuli in $V_1$ are of type (1).
This shows that there exists a string $t_i$ in $B_i$ for all $i$ such that $N(t_i)=B_i$.
Since $A_i$ is a 1-bridge annulus, we can take $t_i$ so that it is also 1-bridge with respect to $h:B^+\to [0,+\infty)$.
Hence $t_i$ bounds a disk $e_i$ with an arc $t_i'$ in $S$ with $\partial t_i=\partial t_i'$ such that $e_i\cap e_j=\emptyset$ for $i\ne j$.
This shows that there exists a $\partial$-compressing disk $E_i$ for $A_i\in \mathcal{A}$ such that $E_i\cap E_j=\emptyset$ for $i\ne j$.
We have particularly $E_1\cap A_i=\emptyset$ for $i=3,5,\ldots,n-1$.

Next suppose that the induction holds for $l=k-1$, and let $l=k$.
Without loss of generality, we may assume that $A_k$ is ``outermost'' in $V_1$, that is, $V_k$ does not contain an annulus in $\partial \mathcal{A}$.
Let $\mathcal{A}_k$ be the set of annuli $A_i\in\mathcal{A}$ which are contained in $B_k$ (including $A_k$), and $\mathcal{A}_k'$ be the set of annuli $A_i\in\mathcal{A}$ which are contained in $V_k$.
Thus $\mathcal{A}=\mathcal{A}_k\cup \mathcal{A}_k'$.
Since $A_k$ is $\partial$-parallel in $V_1$, the height function $h:B^+\to [0,+\infty)$ can be extended to $h':B_k\to [0,+\infty)$, and each annulus $A_i\in \mathcal{A}_k$ is still a 1-bridge annulus with respect to $h'$.
Therefore, by the supposition of the induction for $l=k-1$, there exists a $\partial$-compressing disk $E_1\subset B_k$ for $A_1$ such that $E_1$ intersects each annulus $A_i\in \mathcal{A}_k$ in essential arcs.
Furthermore, $E_1$ can be isotoped so that $\partial E_1$ intersects $A_k\subset \partial B_k$ in essential arcs since $\partial A_k$ consists of mutually parallel loops in $A_1'$ and hence in $\partial B_k\cap V_1$.
On the other hand, since $\mathcal{A}_k'$ does not contain an annulus which is $\partial$-parallel to $A_k'$ in $V_k$, there exists a $\partial$-compressing disk $E_k\subset V_k$ for $A_k$ which is disjoint from $A_i\in \mathcal{A}_k'$ as in the case that $l=0$.
Then, we obtain the desired $\partial$-compressing disk for $A_1$ by gluing $E_1$ and $|\partial E_1\cap A_k|$-copies of $E_k$ along the essential arcs of $\partial E_1\cap A_k$.
\end{proof}

By Lemma \ref{essential arcs}, we have two $\partial$-compressing disks $E_1$ and $E_2$ for $A_1$ and $A_2$ respectively such that $E_j\cap A_i$ consists of essential arcs for $j=1,2$ and $i=3,\ldots n$, and these two disks can be isotoped so that their interiors do not intersect $T^{\pm}$ since $T^{\pm}$ consists of mutually parallel essential arcs or inessential arcs in $A_i$.
This shows that an $n$-bridge decomposition $(B^+,T^+)\cup_{S}(B^-,T^-)$ of $K$ is stabilized.

\subsection{The case that $|F\cap S|=2$}

Next we consider the case that $|F\cap S|=2$.
In this case, $F\cap B^+$ consists of a single annulus $A_1$ and $F\cap B^-$ consists of a single annulus $A_2$ both of which are 1-bridge with respect to $h$.
Therefore the annulus $A_1$ separates $B^+$ into a 3-ball $B_1$ and a solid torus $V_1$, and $A_2$ separates $B^-$ into a 3-ball $B_2$ and a solid torus $V_2$.
It follows that $A_i$ is parallel to $A_i'$ in $V_i$ $(i=1,2)$.
We have $\partial A_1=\partial A_2=C_1\cup C_2$, and $C_1$ and $C_2$ bound disks $D_1$ and $D_2$ in $S$ with $D_1\cap D_2=\emptyset$.

If there exists a string $t_i$ of $T^+$ which is inessential in $A_1$ without loss of generality, then an $n$-bridge decomposition $(B^+,T^+)\cup_{S}(B^-,T^-)$ of $K$ is stabilized.
Indeed, $t_i$ and an arc $t_i'$ in $D_i$ with $\partial t_i=\partial t_i'$ cobound a disk $e_i$ in $B_1$, and an adjacent string $t_{i+1}$ of $T^-$ and an arc $t_{i+1}'$ in $A_2'$ with $\partial t_{i+1}=\partial t_{i+1}'$ cobound a disk $e_{i+1}$ in $V_2$ since $A_2$ is parallel to $A_2'$ in $V_2$.

Otherwise, $T^+$ consists of essential and hence mutually parallel arcs in $A_1$ and $T^-$ also consists of essential and hence mutually parallel arcs in $A_2$.
Since an $n$-bridge position $(B^+,T^+)\cup_{S}(B^-,T^-)$ of $K$ is non-minimal, $T^{\pm}$ consists of $q$-arcs.
In this case, an $n$-bridge decomposition $(B^+,T^+)\cup_{S}(B^-,T^-)$ of $K$ is stabilized as follows.

If we project $T^{\pm}$ into $A_1'=A_2'$ so that it has minimal crossings up to isotopy of $T^{\pm}$ in $A_1'=A_2'$, then we have a $q$-bridge presentation of a $(p,q)$-torus knot.
Then there exists an isotopy of $T^+$ in $B^+$ such that $(q-p)$-pairs of adjacent over/under bridges $(t_i,t_{i+1})$ have no crossing.
Hence we obtain $(q-p)$-destabilizations for $(B^+,T^+)\cup_{S}(B^-,T^-)$.
Furthermore, there exists an isotopy of $T^-$ in $B^-$ which gives a $p$-bridge presentation of a $(q,p)$-torus knot.
For an example on $p=3$ and $q=5$, see Figure \ref{torus knot}.

\begin{figure}[htbp]
	\begin{center}
	\includegraphics[trim=0mm 0mm 0mm 0mm, width=.5\linewidth]{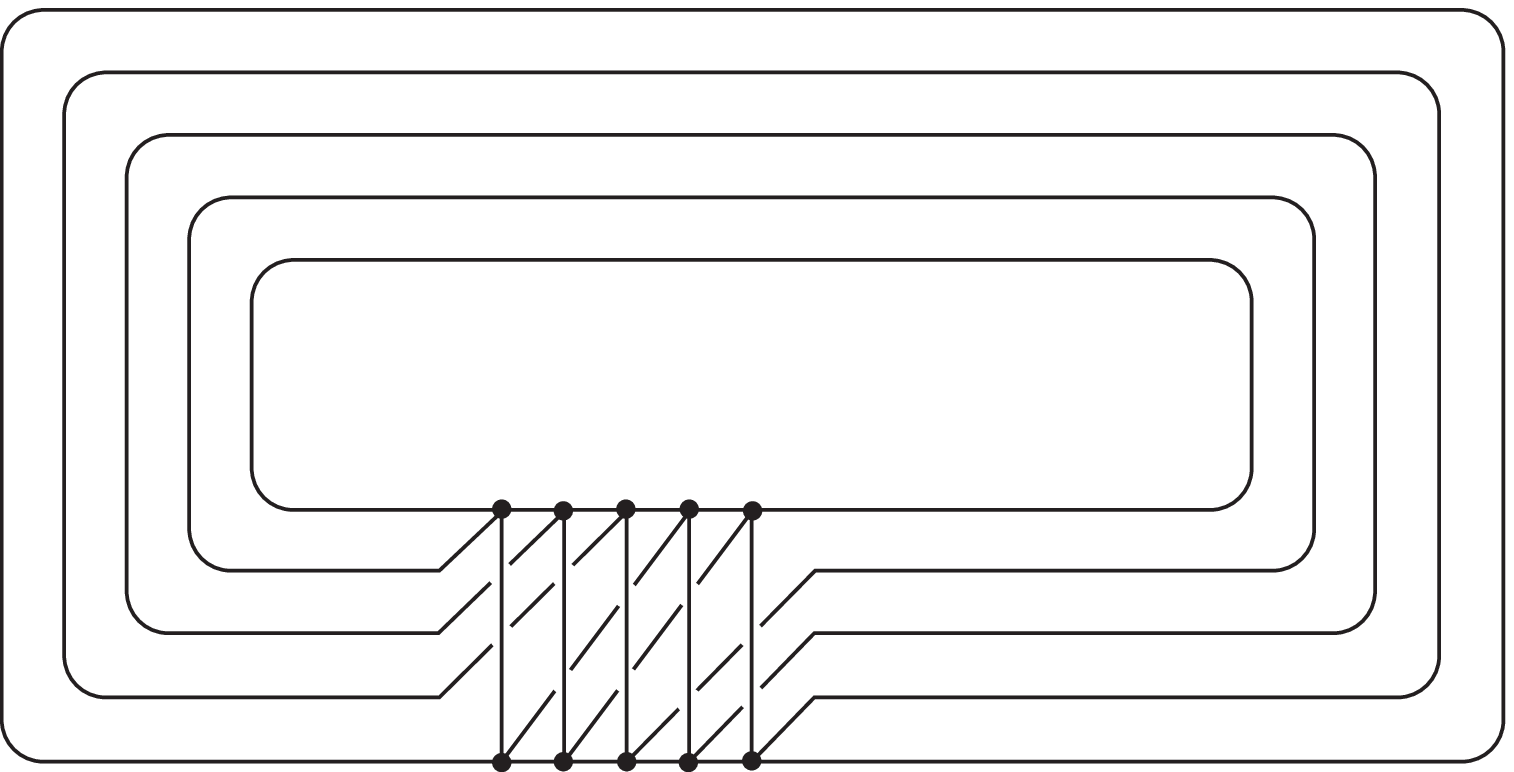}
	\end{center}
	\begin{center}
	\includegraphics[trim=0mm 0mm 0mm 0mm, width=.5\linewidth]{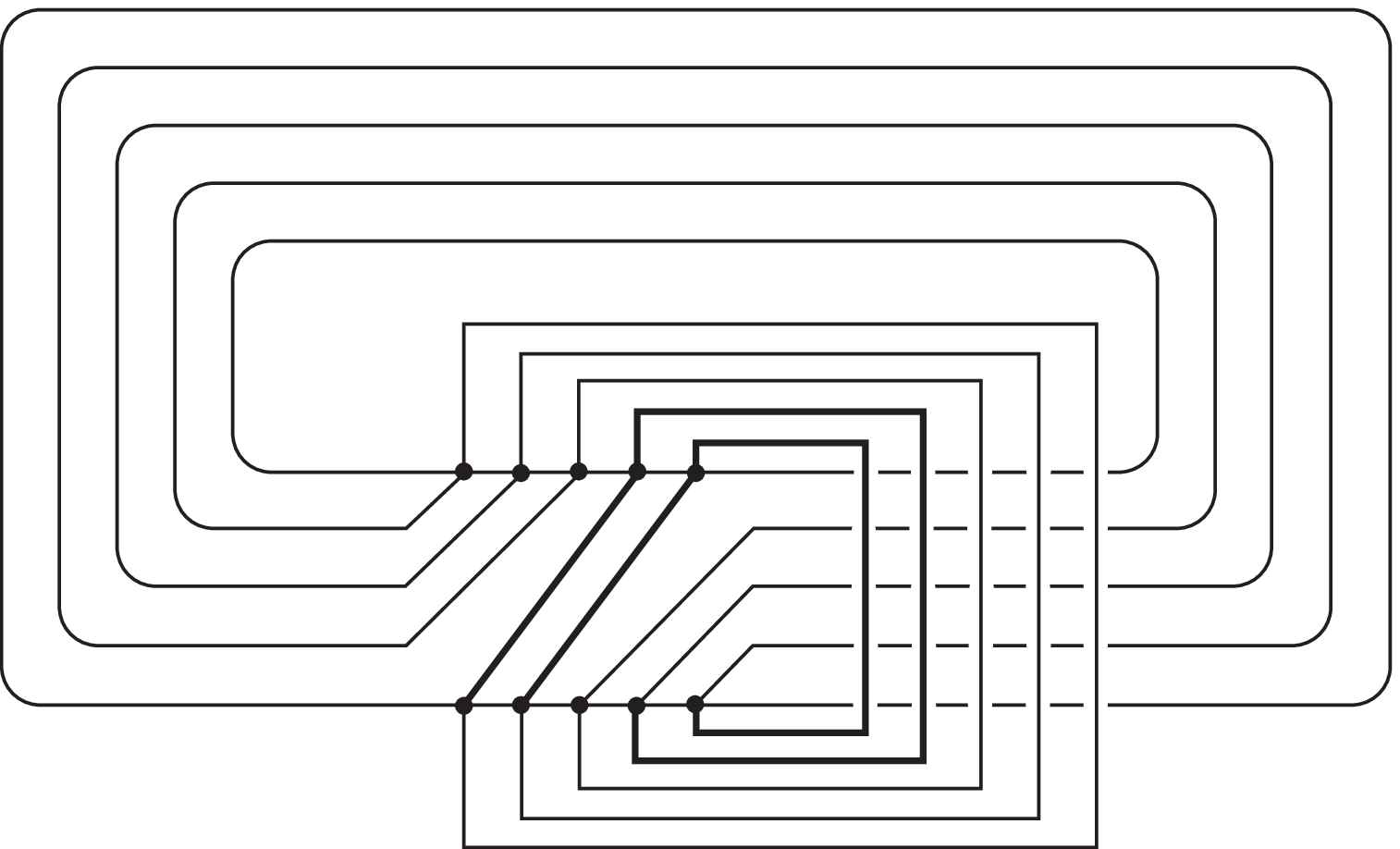}
	\end{center}
	\begin{center}
	\includegraphics[trim=0mm 0mm 0mm 0mm, width=.5\linewidth]{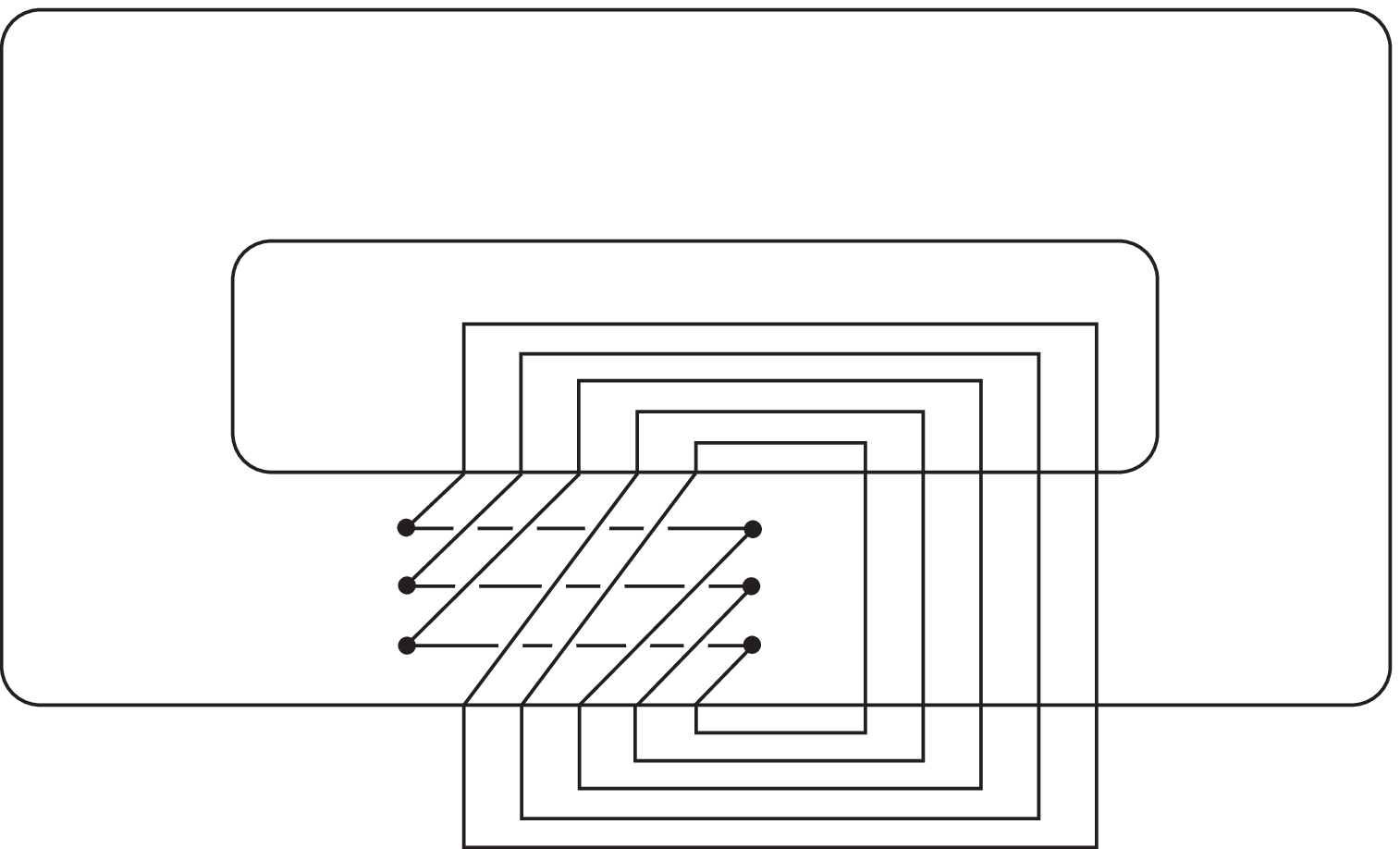}
	\end{center}
	\caption{A deformation from a $(3,5)$-torus knot to a $(5,3)$-torus knot}
	\label{torus knot}
\end{figure}

This completes the proof of Theorem \ref{torus} and shows that minimal bridge decompositions of a torus knot are unique.
\end{proof}

\bigskip


\bibliographystyle{amsplain}

\end{document}